\documentclass[11pt]{amsart}
\usepackage{amsmath, amsfonts, amsthm, amssymb}
\usepackage{url}
\usepackage{natbib}
\usepackage{hyperref}
\usepackage{graphicx}
\usepackage{calc}
\usepackage{subfigure}

\newcommand{\R}{\mathbb{R}}

\newtheorem{thm}{Theorem}
\newtheorem{lem}[thm]{Lemma}
\newtheorem{fact}{Fact}
\newtheorem{cor}[thm]{Corollary}

\begin{document}

\title{Voting in Agreeable Societies}
\date{}
\author[Berg, Norine, Su, Thomas, Wollan]{Deborah E. Berg, Serguei Norine, Francis Edward Su, Robin Thomas, Paul Wollan}
\maketitle

\begin{quote}
{\em My idea of an agreeable person... is a person who agrees with me.}
\newline
\hspace*{3in} ---Benjamin Disraeli \cite{Disr}
\end{quote}

\section{Introduction}
\label{sec:introduction}

When is agreement possible?  An important aspect of group
decision-making is the question of how a group makes a choice when
individual preferences may differ. Clearly, when making a single group
choice, people cannot all have their ``ideal'' preferences, i.e, 
the options that they most desire, if those ideal preferences are
different. However, for the sake of agreement, 
people may be willing to accept as a group choice an
option that is merely ``close'' to their ideal preferences.

Voting is a situation in which people may behave in this way.  The
usual starting model is a one-dimensional political spectrum, 
with conservative positions on the right and liberal positions on the
left, as in Figure \ref{fig:linear}. 
We call each position on the spectrum a {\em platform} that a
candidate or voter may choose to adopt.  
While a voter may represent her ideal platform by some point $x$ on this
line, she might be willing to vote for a candidate who is positioned
at some point ``close enough'' to $x$, i.e., in an interval about $x$.   

\begin{figure}[h]
  \begin{center}
    \scalebox{.9}{\includegraphics{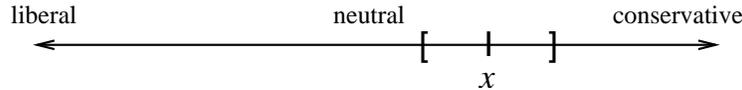}}
  \end{center}
  \caption{A one-dimensional political spectrum, with a single voter's 
  interval of approved platforms.}
  \label{fig:linear}
\end{figure}

In this article, we ask the following: given such preferences on a
political spectrum, when can we guarantee that some fraction 
(say, a majority) of the population will agree on some
candidate?  By ``agree'', we mean in the sense of \emph{approval
voting}, in which voters declare which candidates they find
acceptable.

Approval voting has not yet been adopted for political elections in
the United States.  However, many scientific and mathematical
societies, such as the Mathematical Association of America and the
American Mathematical Society, use approval voting for their
elections.  Additionally, countries other than the United States have
used approval voting or an equivalent system;  for details, see Brams
and Fishburn \cite{brams} who discuss the advantages of approval
voting.

Understanding which candidates can get voter approval can be helpful when there are a large
number of candidates.  An extreme example is the 2003 California
gubernatorial recall election, which had $135$ candidates in the mix
\cite{ca}.  We might imagine these candidates positioned at $135$
points on the line in Figure \ref{fig:linear}, which we think of as a
subset of $\R$.
If each California voter approves of candidates ``close enough'' to
her ideal platform, we may ask under what conditions there is a
candidate that wins the approval of a majority of the voters.

In this setting, we may assume that each voter's set of approved
platforms (her {\em approval set}) is a closed interval in $\R$, and
that there is a set of candidates who take up positions at various
points along this political spectrum.
We shall call this spectrum with a collection of candidates and voters, together
with voters' approval sets, a {\em linear society} (a more precise
definition will be given soon).
We shall say that the linear society is 
{\em super-agreeable} if for every pair of voters
there is some candidate that they would both approve, i.e., each pair of
approval sets contains a candidate in their intersection.
For linear societies this ``local'' condition
guarantees a strong ``global'' property, namely, that there is a
candidate that \emph{every} voter approves!  As we shall see in
Theorem \ref{cor:super-agreeable}, this can be viewed as a consequence of 
Helly's theorem about intersections of convex sets.

But perhaps this is too strong a conclusion.  Is there a weaker 
local condition that would guarantee that only 
a majority (or some other fraction) of the voters would 
approve a particular candidate?
For instance, we relax the condition above and call 
a linear society \emph{agreeable} if 
among every three voters, some pair of voters approve the same candidate.  
Then it is not hard to show:

\begin{thm}
\label{thm:agree-lin-soc}
In an agreeable linear society, there is a candidate who has the
approval of at least half the voters.
\end{thm}

More generally, call a linear society
\emph{$(k,m)$-agreeable} if it has at least $m$ voters, 
and among every $m$ voters, some subset of $k$ voters approve the same
candidate.  Then our main theorem is a generalization of the previous
result:

\begin{thm}[The Agreeable Linear Society Theorem]
\label{thm:km-agree} 
Let $2\le k\le m$.  
In a $(k,m)$-agreeable linear society of $n$ voters, there is a candidate who has
the approval of at least $n(k-1)/(m-1)$ of the voters.
\end{thm}

We prove a slightly more general
result in Theorem~\ref{thm:clique} and also briefly study societies
whose approval sets are convex subsets of $\R^d$.

As an example, consider a city with 
fourteen restaurants along its main boulevard:
$$
A\ B\ C\ D\ E\ F\ G\ H\ I\ J\ K\ L\ M\ N
$$
and suppose every resident dines only at the five restaurants 
closest to his/her house (a set of consecutive restaurants, e.g., $DEFGH$).
A consequence of Theorem \ref{thm:agree-lin-soc} is that there must be
a restaurant that is patronized by at least half the residents.  Why?
The pigeonhole principle guarantees that among every 3 residents, each
choosing 5 of 14 restaurants, there must be a restaurant approved by
at least 2 of them; hence this linear society is agreeable 
and Theorem \ref{thm:agree-lin-soc} applies.
For an example of Theorem \ref{thm:km-agree}, see Figure \ref{fig:candidates}, which shows  
a $(2,4)$-agreeable linear society, 
and indeed there are candidates that receive at least $1/3$ of the votes (in this case $\lceil 7/3 \rceil = 3$).

We shall begin with some definitions, and explain connections to  classical convexity theorems, graph colorings, and maximal cliques in graphs.  Then we prove Theorem \ref{thm:km-agree}, discuss extensions to higher-dimensional spectra, and conclude with some questions for further study.

\section{Definitions}%
\label{sec:defs}

In this section, we fix terminology and explain 
the basic concepts upon which our results rely.  Let us suppose that the
set of all possible preferences is modeled by a set $X$, called the
\emph{spectrum}.  Each element of the spectrum is a \emph{platform}.
Assume that there is a finite set $V$ of \emph{voters}, and each voter
$v$ has an \emph{approval set} $A_v$ of platforms.

We define a {\em society} $S$ to be a triple $(X, V, \mathcal A)$
consisting of a spectrum $X$, a set of
voters $V$, and a collection $\mathcal A$ of approval sets for all the
voters.  Of particular interest to us will be the case of a 
{\em linear society}, in which $X$ is a closed subset of $\R$ 
and approval sets in $\mathcal A$ are of the form $X \cap I$ where $I$
is either empty or a closed bounded interval in $\R$.
In general, however, $X$ could be any set and the collection $\mathcal A$
of approval sets could
be any class of subsets of $X$.  
In Figure~\ref{fig:shaded} we illustrate a linear society, 
where for ease of display we have separated the
approval sets vertically so that they can be distinguished.

\begin{figure}[thb]
  \begin{center}
    \scalebox{1}{\includegraphics{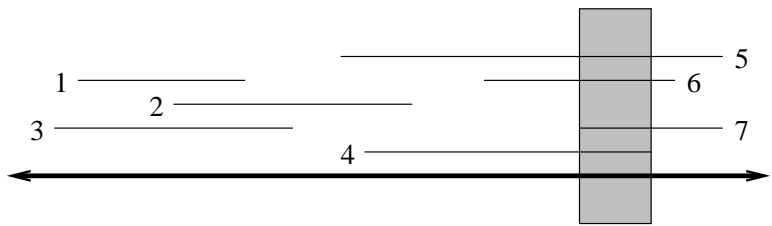}}
  \end{center}
  \caption{A linear society with infinite spectrum: each interval
  (shown here displaced above the spectrum) 
  corresponds to the approval set of a voter.  The shaded region
  indicates platforms with agreement number 4.  
  This is a $(2,3)$-agreeable society.}
  \label{fig:shaded}
\end{figure}

Our motivation for considering intervals as approval sets arises from
imagining that voters have an ``ideal'' platform along a linear scale
(similar to Coombs' $J$-scale \cite{coombs}), and that voters are willing
to approve ``nearby'' platforms, yielding approval sets that are
connected intervals. Unlike the Coombs scaling theory, however, we are not
concerned with the order of preference of approved platforms; all platforms within
a voter's approval set have equivalent status as ``approved'' by that
voter.  
We also note that while we model our linear scale as a subset of $\R$,
none of our results about linear societies depends on the metric;
we only appeal to the ordinal properties of $\R$.

We have seen that politics provides natural examples of linear
societies.  For a different example, $X$ could represent a 
temperature scale, $V$ a set of people that live in a house, and each
$A_v$  a range of temperatures that 
person $v$ finds comfortable.  
Then one may ask: at what temperature should the thermostat be set so as 
to satisfy the largest number of people?

\begin{figure}[htb]
  \begin{center}
    \scalebox{1}{\includegraphics{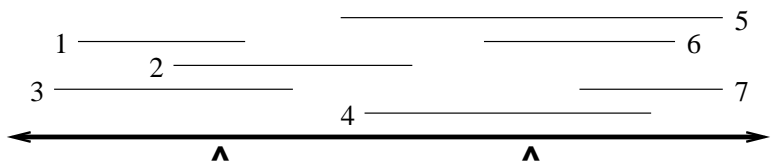}}
  \end{center}
  \caption{A linear society with a spectrum of two candidates (at
    platforms marked by carats): take the approval sets of the society
    of Figure \ref{fig:shaded} and intersect with these candidates.
    It is a $(2,4)$-agreeable linear society.}
  \label{fig:candidates}
\end{figure}

Two special cases of linear societies are worth mentioning.  When
$X=\R$ we should think of $X$ as an infinite spectrum of platforms
that potential candidates might adopt.  However, in practice there are
normally only finitely many candidates.  
We model that situation by
letting $X$ be the set of platforms adopted by actual candidates.
Thus one could think of $X$ as either the set of all platforms, or the
set of (platforms adopted by) candidates.
See Figures \ref{fig:shaded} and \ref{fig:candidates}.

Let $1\leq k\leq m$ be integers.  
Call a society $(k,m)$-\emph{agreeable}
if it has at least $m$ voters, and for any subset of $m$
voters, there is at least one platform that at least $k$ of them can
agree upon, i.e., there is a point common to at least $k$ of the
voters' approval sets.  Thus to be  $(2,3)$-agreeable 
is the same as to be {\em agreeable}, 
and to be $(2,2)$-agreeable is the same as to be
{\em super-agreeable}, as defined earlier.

One may check that the society of Figure \ref{fig:shaded} is 
$(2,3)$-agreeable.  It is not $(3,4)$-agreeable, however, 
because among voters $1,2,4,7$ no three of
them share a common platform.  The same society, after restricting the
spectrum to a set of candidates, is
the linear society shown in Figure \ref{fig:candidates}.
It is not $(2,3)$-agreeable, because among voters
$2,4,7$ there is no pair that can agree on a candidate  
(in fact, voter $7$ does not approve any candidate).  However, one may
verify that this linear society is $(2,4)$-agreeable.

For a society $S$, the \emph{agreement number} of a
platform, $a(p)$, is the number of voters in $S$ who
approve of platform $p$.  
The \emph{agreement number} $a(S)$ of a society 
$S$ is the maximum agreement number over all platforms in the
spectrum, i.e., 
$$a(S) = \max_{p \in X} a(p).$$
The \emph{agreement proportion} of $S$ is simply the agreement number of
$S$ divided by the number of voters of $S$.  This concept is
useful when we are interested in percentages of the population
rather than the number of voters. 
The society of Figure~\ref{fig:shaded} has agreement number~4, which 
can be seen where the 
shaded rectangle covers the set of platforms that have maximum agreement number.

%--------------------------------------
\section{Helly's Theorem and Super-Agreeable Societies}
\label{sec:helly}

Let us say that a society is $\R^d$-\emph{convex} if the spectrum is
$\R^d$ and each approval set is a closed convex subset of $\R^d$.
Note that an $\R^1$-convex society is a linear society with spectrum
$\R$.  
An $\R^d$-convex society can arise when considering a multi-dimensional
spectrum, such as when evaluating political platforms over several
axes (e.g., conservative vs.~liberal, pacifist vs.~militant, interventionist
vs.~isolationist).  Or, the
spectrum might be arrayed over more personal dimensions: the dating
website \textit{eHarmony} claims to use up to 29 of them \cite{eharmony}.
In such situations, the convexity of approval
sets might, for instance, follow from an independence-of-axes
assumption and convexity of approval sets along each axis.  

To find the agreement proportion of an $\R^d$-convex society, we 
turn to work concerning intersections of convex sets.
The most well-known result in this area is Helly's theorem.
This theorem was proven by Helly
in 1913, but the result was not published until 1921, by Radon
\cite{radon}.

\begin{thm}[Helly]
Given $n$ convex sets in $\R^d$ where $n > d$, if every $d + 1$ of
them intersect at a common point, then they all intersect at a
common point.
\end{thm}

Helly's theorem has a nice interpretation for $\R^d$-convex societies:

\begin{cor}
\label{thm:22n} For every $d\ge 1$, a $(d+1,d+1)$-agreeable
$\R^d$-convex society must contain at least one platform that is
approved by all voters.  
\end{cor}

Notice that for the corollary to hold for $d>1$ it is important that the
spectrum of an $\R^d$-\emph{convex} society be all of $\R^d$.
However, for $d=1$ that is not necessary, as we now show.

\begin{thm}[The Super-Agreeable Linear Society Theorem]
\label{cor:super-agreeable}
A super-agreeable linear society must contain at 
least one platform that is approved by all voters.
\end{thm}

We provide a simple proof of this theorem, since the result will be needed 
later.  When the spectrum is all of $\R$, this theorem is just Helly's
theorem for $d=1$;
a proof of Helly's theorem for general $d$ may be found in
\cite{Matousek}.

\begin{proof}
Let $X \subseteq \R$ denote the spectrum.
Since each voter $v$ agrees on at least one platform with every
other voter, we see that the approval sets $A_v$ must be nonempty.
Let $L_v = \min A_v$, $R_v = \max A_v$, and let
$x = \max_v\{L_v\}$ and $y = \min_v\{R_v\}$.
The first two minima/maxima exist because each $A_v$ is compact; the
last two exist because the number of voters is finite.

\begin{figure}[htb]
  \begin{center}
    \scalebox{1}{\includegraphics{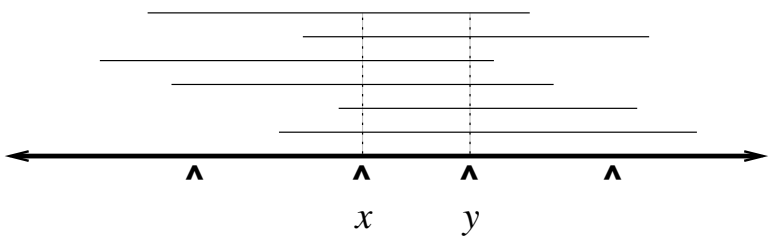}}
  \end{center}
  \caption{A super-agreeable linear society of 6 voters and 4 candidates, with
  agreement number 6.}
  \label{fig:226}
\end{figure}

We claim that $x \leq y$.  Why?  Since every pair of approval sets 
intersect in some platform, we see that 
$L_i \leq R_j$ for every pair of voters $i,j$.
In particular, let $i$ be the voter whose $L_i$ 
is maximal and let $j$ be the voter whose $R_j$ is minimal. 
Hence $x \leq y$ and every approval set contains all platforms of $X$
that are in 
the nonempty interval $[x,y]$, and in particular, the platform $x$.
\end{proof}

The idea of this proof can be easily extended to furnish a proof of 
Theorem \ref{thm:agree-lin-soc}.
\medskip

\noindent
{\em Proof of Theorem \ref{thm:agree-lin-soc}.}
Using the same notations as in the prior proof, 
if $x \leq y$ then that proof 
shows that every approval set contains the platform $x$.  Otherwise
$x > y$ implies $L_i > R_j$ so that $A_j$ and $A_i$ do not 
contain a common platform.

We claim that for any other voter $v$, the approval set $A_v$
contains either platform $x$ or $y$ (or both).  
After all, the society is agreeable, so
some pair of $A_i,A_j,A_v$ must contain a
common platform; by the remarks above it must be that $A_v$
intersects one of $A_i$ or $A_j$.  If $A_v$ does not contain $x=L_i$
then since $L_v \leq L_i$ (by definition of $x$), we must have that 
$R_v < L_i$ and $A_v \cap A_i$ does not contain a platform. 
Then $A_v \cap A_j$ must contain a platform; hence $L_v \leq R_j$.
Since $R_j \leq R_v$ (by definition of $y$), the platform $y=R_j$ must be
in $A_v$.

Thus every approval set contains either $x$ or $y$, and by the
pigeonhole principle one of them must be contained in at least half the
approval sets.~\qed
\medskip

Proving the more general Theorem \ref{thm:km-agree} 
will take a little more work.

%--------------------------------------
\section{The Agreement Graph of Linear Societies is Perfect}%
\label{sec:midgraphs}

To understand $(k,m)$-agreeability, it will be helpful to use a graph
to represent the intersection relation on  approval sets.
Recall that a 
\emph{graph} $G$ consists of a finite set $V(G)$ of \emph{vertices}
and a set $E(G)$ of $2$-element subsets of $V(G)$, called \emph{edges}. If
$e=\{u,v\}$ is an edge, then we say that $u,v$ are the \emph{ends}
of $e$, and that $u$ and $v$ are \emph{adjacent} in $G$.  We use
$uv$ as shorthand notation for the edge $e$.

Given a society $S$, we construct the \emph{agreement graph} $G$ of
$S$ by letting the vertices $V(G)$ be the voters of $S$ and the edges
$E(G)$ be all pairs of voters $u,v$ whose approval sets intersect each
other.  Thus $u$ and $v$ are connected by an edge 
if there is a platform that both $u$ and $v$ would approve.
Note that the agreement graph of a society with agreement number equal
to the number of voters is a complete graph (but the converse
is false in higher dimensions, as we discuss later).
Also note that a vertex $v$ is isolated if $A_v$ is empty or 
disjoint from other approval sets.

\begin{figure}[htb]
  \begin{center}
    \scalebox{1}{\includegraphics{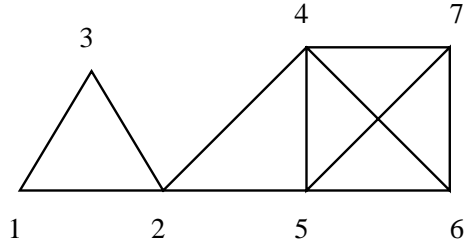}}
  \end{center}
  \caption{The agreement graph for the society in Figure \ref{fig:shaded}.  Note that voters $4,5,6,7$ form a maximal clique that corresponds to the maximal agreement number in Figure \ref{fig:shaded}.}
  \label{fig:7agree}
\end{figure}

The \emph{clique number} of $G$, written $\omega(G)$, is the
greatest integer $q$ such that $G$ has a set of $q$ pairwise
adjacent vertices, called a \emph{clique} of size $q$.
By restricting our attention to members of a clique, and applying
the Super-Agreeable Linear Society Theorem,
we see that there is a platform that has the approval of every member 
of a clique, and hence:

\begin{fact}
\label{fac1}
For the agreement graph of a linear society, 
the clique number of the graph is the agreement number of the society.
\end{fact}

This fact does not necessarily hold if the society is not linear.
For instance, 
it is easy to construct an $\R^2$-convex society with three voters 
such that every two voters agree on a platform, but all three of them
do not.  It does, however hold in $\R^d$ for {\em box societies}, to be discussed in Section \ref{sec:convsoc}.

Now, to get a handle on the clique number, we shall make a connection 
between the clique number and colorings of the agreement graph.
The \emph{chromatic number} of $G$, written $\chi(G)$, is
the minimum number of colors necessary to color the vertices of $G$
such that no two adjacent vertices have the same color.  
Thus two voters may have the same color as long as they do not agree
on a platform.  Note that in all cases, $\chi(G) \geq \omega(G)$.

A graph $G$ is called an \emph{interval graph} if we can assign to every vertex $x$
a closed interval or an empty set $I_x\subseteq \R$ such that $xy \in E(G)$ if
and only if $I_x \cap I_y \neq \emptyset$.  We have:

\begin{fact}\label{fac2}
\rm The agreement graph of a linear society is an interval graph.
\end{fact}

To see that Fact \ref{fac2} holds let the linear society be
$(X,V,\mathcal A)$, and let the voter approval sets be $A_v=X \cap
I_v$, where $I_v$ is a closed bounded interval or empty.  We may
assume that each $I_v$ is a minimal closed interval satisfying 
$A_v = X \cap I_v$; then the intervals $\{ I_v: v\in V \}$ provide an
interval representation of the agreement graph, as desired.

An \emph{induced subgraph} of a graph $G$ is a graph $H$ such that
$V(H)\subseteq V(G)$ and the edges of $H$ are the edges of $G$ 
that have both ends in $V(H)$. 
If every induced subgraph $H$ of a graph $G$ satisfies
$\chi(H) = \omega(H)$, then $G$ is called a \emph{perfect graph}; see, e.g., 
\cite{RamRee}.
The following is a standard fact \cite{West} about interval graphs:

\begin{thm}
\label{thm:perfect}
Interval graphs are perfect.
\end{thm}

\begin{proof}
Let $G$ be an interval graph, and for $v\in V(G)$, let $I_v$ be the
interval representing the vertex $v$. Since every induced subgraph
of an interval graph is an interval graph, it suffices to show that
$\chi (G)=\omega$, where $\omega=\omega (G)$.
We proceed by
induction on $|V(G)|$. The assertion holds for the null graph, and
so we may assume that $|V(G)|\ge 1$, and that the statement holds
for all smaller graphs. Let us select a vertex $v\in V(G)$ such that
the right end of $I_v$ is as small as possible.  It follows that
the elements of
$N$, the set of neighbors of $v$ in $V(G)$, are pairwise adjacent because their
intervals must all contain the right end of $I_v$, and hence
$|N|\le\omega-1$.  See Figure \ref{fig:iviw}.
By the inductive hypothesis, the graph $G\backslash \{v\}$ obtained
from $G$ by deleting $v$ can be colored using $\omega$ colors, and
since $v$ has at most $\omega-1$ neighbors, this coloring can be
extended to a coloring of $G$, as desired.
\end{proof}

\begin{figure}[htb]
  \begin{center}
    \scalebox{.5}{\includegraphics{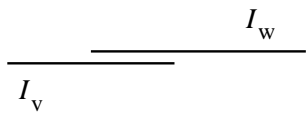}}
  \end{center}
  \caption{If $I_v, I_w$ intersect and the 
      right end of $I_v$ is smaller than the right end of $I_w$, 
      then $I_w$ must contain the right end of $I_v$.}
  \label{fig:iviw}
\end{figure}

The perfect graph property will allow us, in the next section, to make
a crucial connection between the $(k,m)$-agreeability condition and
the agreement number of the society.  Given its importance in our
setting, it is worth making a few comments about how perfect graphs
appear in other contexts in mathematics, theoretical computer science,
and
operations research.  The concept was introduced in 1961 by
Berge~\cite{BerSPGC}, who was motivated by a question in
communication theory, specifically, the determination of the Shannon
capacity of a graph \cite{Shannon}. Chv\'atal later discovered that a
certain class of linear programs always have an integral solution if
and
only if the corresponding matrix arises from a perfect graph in a
specified
way \cite{ChvPoly, RamRee, ThoPerfSurv}.  As pointed out in
\cite{RamRee}, algorithms to solve semi-definite programs grew out of
the theory of perfect graphs.  It has been proven recently
\cite{ChuRobSeyThoSPGC} that a graph is perfect if and only if it has
no induced subgraph isomorphic to a cycle of odd length at least five,
or a complement of such a cycle.

%--------------------------------------
\section{$(k,m)$-Agreeable Linear Societies}
\label{sec:km}

We now use the connection between perfect graphs, the clique
number, and the chromatic number to obtain a lower bound
for the agreement number of a $(k,m)$-agreeable linear society
(Theorem \ref{thm:clique}).
We first need a lemma that says that
in the corresponding agreement graph, the $(k,m)$-agreeable condition 
prevents any coloring of the graph from having too many vertices of
the same color.  Thus, there must be many colors and, since the graph is
perfect, the clique number must be large as well.

\begin{lem}\label{nlrrho}
Given integers $m\geq k\geq 2$, 
let positive integers $q, \rho$ be defined by
the division with remainder: $m-1 =(k-1)q+\rho$, where
$0 \leq \rho \leq k-2$.
Let $G$ be a graph on $n\ge m$ vertices
with chromatic number $\chi$ such that every subset of $V(G)$ of size
$m$ includes a clique of size $k$.  Then $n\leq \chi q+\rho$, or 
$\chi \geq (n-\rho)/q$.
\end{lem}

\begin{proof} Let the graph be colored using the colors $1,2,\dots, \chi$,
and for $i=1,2,\dots,\chi$ let $C_i$ be the set of vertices of $G$
colored $i$.  We may assume, by permuting the colors, that $|C_1|\ge
|C_2|\ge\cdots\ge |C_\chi|$.  Since $C_1\cup C_2\cup\cdots \cup
C_{k-1}$ is colored using $k-1$ colors, it includes no clique of
size $k$, and hence, $|C_1\cup C_2\cup\cdots \cup C_{k-1}|\le m-1$.
It follows that $|C_{k-1}|\le q$, for otherwise $|C_1\cup C_2\cup
\cdots \cup C_{k-1}|\ge (k-1)(q+1)\ge (k-1)q+\rho +1=m$, a
contradiction. Thus $|C_i|\leq q$ for each $i\geq k$ and
$$
n=\sum^{k-1}_{i=1}|C_i|+\sum^\chi_{i=k}|C_i|\le m-1+(\chi-k+1)q=(k-1)q+
\rho +(\chi-k+1)q=\chi q+\rho,$$
as desired.
\end{proof}

\begin{thm}
\label{thm:clique} 
Let $2\le k\le m$.  If $G$ is the
agreement graph of a $(k,m)$-agreeable linear society and 
$q, \rho$ are defined by the division with remainder: 
$m-1 =(k-1)q+\rho$, $\rho \leq k-2$, then the clique number satisfies:
$$\omega(G)\geq \lceil (n-\rho)/q\rceil,$$
and this bound is best possible.
It follows that this is also a lower bound on the agreement number,
and hence every linear $(k,m)$-agreeable society has agreement
proportion at least $(k-1)/(m-1)$.
\end{thm}

\begin{proof}
By Fact~\ref{fac2} and Theorem~\ref{thm:perfect} the graph $G$ is
perfect. Thus the chromatic number of $G$ is equal to $\omega (G)$,
and hence $\omega (G)\ge \lceil (n-\rho)/q \rceil$ by Lemma~\ref{nlrrho}, as
desired.  

The second assertion follows from Fact~\ref{fac1} and noting
that 
$(n-\rho)(m-1) = n(k-1)q+n\rho - \rho(m-1)=n(k-1)q + \rho(n-m+1)\geq
n(k-1)q$,
from which we see that $(n-\rho)/q\ge n(k-1)/(m-1)$.

Let us observe that the bound $\lceil (n-\rho)/q\rceil$ in
Theorem~\ref{thm:clique} is best possible.  Indeed, let
$I_1,I_2,\dots, I_q$ be disjoint intervals, for $i=q+1,q+2,\dots,
n-\rho$ let $I_i=I_{i-q}$, and let $I_{n-\rho+1}, I_{n-\rho+2},
\dots, I_n$ be pairwise disjoint and disjoint from all the previous
intervals, e.g., see Figure \ref{fig:optdisjoint}.
Then the society with approval sets $I_1,I_2,\dots, I_n$
is $(k,m)$-agreeable and its agreement graph has clique number
$\lceil (n-\rho)/q\rceil$.
\end{proof}

\begin{figure}[htb]
  \begin{center}
    \scalebox{1}{\includegraphics{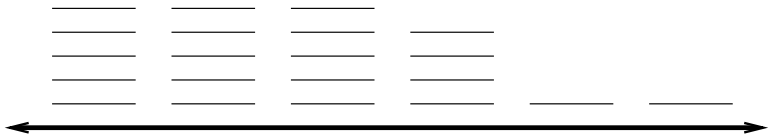}}
  \end{center}
  \caption{A linear $(4, 15)$-society with $n=21$ voters.  Here $q=4$
     and  $\rho=2$, so
     the clique number is at least $\lceil (n-\rho)/q\rceil = 5$.}
  \label{fig:optdisjoint}
\end{figure}

The Agreeable Linear Society Theorem (Theorem
\ref{thm:km-agree}) now follows as a corollary of 
Theorem \ref{thm:clique}.

\section{$\R^d$-convex and $d$-box Societies}%
\label{sec:convsoc} 
In this section we prove a higher-dimensional
analogue of Theorem~\ref{thm:clique} by giving a lower bound on the
agreement proportion of a $(k,m)$-agreeable $\R^d$-convex society.
We need a different method than our method for $d = 1$, because for
$d\ge 2$, neither Fact~\ref{fac1} nor Fact~\ref{fac2} holds. 

Also, we remark that, unlike our results on linear societies,
our results in this section about the agreement
proportion for platforms 
will not necessarily hold when restricting
the spectrum to a finite set of candidates in $\R^d$.

We shall use
the following generalization of Helly's theorem, due to
Kalai~\cite{KalHelly}.

\begin{thm}[The Fractional Helly's Theorem]
\label{fractionalhelly} Let $d\ge 1$ and $n\ge d+1$ be integers, let
$\alpha \in [0,1]$ be a real number, and let $\beta =
1-(1-\alpha)^{1/(d+1)}$.  Let $F_1,F_2,\dots, F_n$ be convex sets in
$\R^d$ and assume that for at least $\alpha\binom n{d+1}$ of the
$(d+1)$-element index sets $I\subseteq \{1,2,\dots, n\}$ we have
$\bigcap_{i\in I} F_i\ne\emptyset$. Then there exists a point in
$\R^d$ contained in at least $\beta n$ of the sets $F_1,F_2,\dots,
F_n$.
\end{thm}

The following is the promised analogue of Theorem~\ref{thm:clique}.

\begin{thm}
\label{dconvex}
Let $d\ge 1$, $k\ge 2$ and $m\ge k$ be integers, where $m>d$.  Then every
$(k,m)$-agreeable $\R^d$-convex society has agreement proportion at
least $1-\left(1-\binom k{d+1}\big\slash\binom
m{d+1}\right)^{1/(d+1)}$.
\end{thm}

\begin{proof}
Let $S$ be a $(k,m)$-agreeable $\R^d$-convex society, and let
$A_1,A_2,\dots, A_n$ be its voter approval sets. Let us call a set
$I\subseteq \{1,2,\dots, n\}$ \emph{good} if $|I|=d+1$ and
$\bigcap_{i\in I}A_i\ne\emptyset$.  By Theorem~\ref{fractionalhelly}
it suffices to show that there are at least 
$\binom k{d+1}\binom n{d+1}\big\slash \binom m{d+1}$ 
good sets.  We will do this by counting
in two different ways the number $N$ of all pairs $(I,J)$, where
$I\subseteq J\subseteq \{1,2,\dots, n\}$, $I$ is good, and $|J|=m$.
Let $g$ be the number of good sets.  Since every good set 
is of size $d+1$ and extends to
an $m$-element subset of $\{1,2,\dots, n\}$ in
$\binom{n-d-1}{m-d-1}$ ways, we have $N=g\binom{n-d-1}{m-d-1}$.  On
the other hand, every $m$-element set $J\subseteq\{1,2,\dots, n\}$
includes at least one $k$-element set $K$ with $\bigcap_{i\in K}
A_i \ne\emptyset$ (because $S$ is $(k,m)$-agreeable), and $K$ in turn
includes $\binom k{d+1}$ good sets.  Thus $N\ge \binom k{d+1}\binom
nm$, and hence $g\ge \binom k{d+1}\binom n{d+1}\big\slash \binom
m{d+1}$, as desired.
\end{proof}

For $d=1$, Theorem~\ref{dconvex} gives a worse bound than
Theorem~\ref{thm:clique}, and hence for $d\ge 2$, the bound is most
likely not best possible.  However, a possible improvement must use
a different method, because the bound in
Theorem~\ref{fractionalhelly} \emph{is} best possible.

A \emph{box} in $\R^d$ is the Cartesian product of $d$ closed
intervals, and we say that a society is a $d$-\emph{box society} if
each of its approval sets is a box in $\R^d$.  By projection onto each
axis, it follows from
Theorem~\ref{cor:super-agreeable} 
that $d$-box societies satisfy the conclusion of 
Fact~\ref{fac1} (namely, that the clique number equals the agreement
number), and hence their agreement graphs capture all the
essential information about the society.  Unfortunately, agreement
graphs of $d$-box societies are, in general, not perfect.  For
instance, there is a 2-box society 
(Figure \ref{fig:5cycle}) whose agreement graph is the
cycle on five vertices; hence its chromatic number is 3 but its clique
number is 2.

For $k\le m\le 2k-2$, the 
following theorem
will resolve the agreement proportion problem for all
$(k,m)$-agreeable societies satisfying the conclusion of 
Fact~\ref{fac1}, and hence for all $(k,m)$-agreeable $d$-box societies 
where $d\ge 1$ (Theorem \ref{dbox}).

\begin{figure}[htb]
  \begin{center}
    \scalebox{.50}{\includegraphics{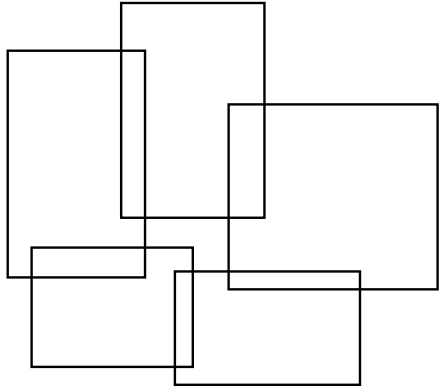}}
  \end{center}
  \caption{A 2-box society whose agreement graph is a 5-cycle.}
  \label{fig:5cycle}
\end{figure}

\begin{thm}\label{smallm}
Let $m,k\ge 2$ be integers with $k\le m\le 2k-2$, and let $G$ be a
graph on $n\ge m$ vertices such that every subset of $V(G)$ of size
$m$ includes a clique of size $k$.  Then $\omega (G)\ge n-m+k$.
\end{thm}

Before we embark on a proof let us make a few comments.  First of
all, the bound $n-m+k$ is best possible, as shown by the graph
consisting of a clique of size $n-m+k$ and $m-k$ isolated vertices.
Second, the conclusion $\omega (G)\ge n-m+k$ implies that every
subset of $V(G)$ of size $m$ includes a clique of size $k$, and so
the two statements are equivalent under the hypothesis that $k\le
m\le 2k-2$.  Finally, this hypothesis is necessary, because if $m\ge
2k-1$, then for $n\ge 2(m-k+1)$, the disjoint union of cliques of
sizes $\lfloor n/2\rfloor$ and $\lceil n/2\rceil$ satisfies the
hypothesis of Theorem~\ref{smallm}, but not its conclusion.

A \emph{vertex cover} of a graph $G$ is a set $Z\subseteq V(G)$ such
that every edge of $G$ has at least one end in $Z$.  We say a set
$S\subseteq V(G)$ is \emph{stable} if no edge of $G$ has both ends
in $S$. We deduce Theorem~\ref{smallm} from the following lemma.

\begin{lem}\label{coverlemma}
Let $G$ be a graph with vertex cover of size $z$ and none of size $z-1$ such that
$G\backslash\{v\}$ has a vertex cover of size at most $z-1$ for all
$v\in V(G)$. Then $|V(G)|\le 2z$.
\end{lem}

\begin{proof} Let $Z$ be a vertex cover of $G$ of size $z$.
For every $v\in V(G)-Z$ let $Z_v$ be a vertex cover of $G\backslash
\{v\}$ of size $z-1$, and let $X_v=Z-Z_v$.  Then $X_v$ is a stable set.
For $X\subseteq Z$ let $N(X)$ denote the set of neighbors of $X$
outside $Z$.  We have $v\in N(X_v)$ and $N(X_v)-\{v\}\subseteq
Z_v-Z$, and so
$$|X_v|=|Z-Z_v| = |Z|-|Z\cap Z_v|=|Z_v|+1-|Z\cap Z_v|=
|Z_v-Z|+1\ge |N(X_v)|.$$
On the other hand, if $X\subseteq Z$ is stable, then $|N(X)|\ge |X|$,
for otherwise $(Z-X)\cup N(X)$ is a vertex cover of $G$ of size
at most $z-1$, a contradiction.  We have
\begin{equation}\label{eq:contra}
|Z|\ge |\bigcup X_v|\ge |\bigcup N(X_v)|\ge |V(G)|-|Z|,
\end{equation}
where both unions are over all $v\in V(G)-Z$, and hence $|V(G)|\le 2z$,
as required.  To see that the second inequality holds let $u,v\in V(G)-Z$.
Then
\begin{align*}
|X_u\cup X_v|&=|X_u|+|X_v|-|X_u\cap X_v|\ge
|N(X_u)|+|N(X_v)|-|N(X_u\cap X_v)|\\
&\ge |N(X_u)| + |N(X_v)|-|N(X_u)\cap N(X_v)|=|N(X_u)\cup N(X_v)|,
\end{align*}
and, in general, the second inequality of \eqref{eq:contra}
follows by induction on $|V(G)-Z|$.
\end{proof}

\noindent{\em Proof of Theorem~\ref{smallm}.}  We proceed by
induction on $n$.  If $n=m$, then the conclusion certainly holds,
and so we may assume that $n\ge m+1$ and that the theorem holds for
graphs on fewer than $n$ vertices.  We may assume that $m>k$, for
otherwise the hypothesis implies that $G$ is the complete graph.  We
may also assume that $G$ has two nonadjacent vertices, say $x$ and
$y$, for otherwise the conclusion holds. 
Then in $G$, every clique contains at most one of $x,y$, so 
in the graph $G\backslash
\{x,y\}$ every set of vertices of size $m-2$ includes a clique of
size $k-1$. Since $k-1\le m-2\le 2(k-1)-2$ we deduce by 
the inductive hypothesis 
that $\omega (G)\ge \omega (G\backslash \{x,y\})\ge
n-2-(m-2)+k-1=n-m+k-1$.  We may assume in the last statement 
that equality holds
throughout, because otherwise $G$ satisfies the conclusion of the
theorem. Let $\bar G$ denote the complement of $G$; that is, the
graph with vertex set $V(G)$ and edge set consisting of precisely
those pairs of distinct vertices of $G$ that are not adjacent in
$G$. Let us notice that a set $Q$ is a clique in $G$ if and only if
$V(G)-Q$ is a vertex cover in $\bar G$. 
Thus $\bar G$ has a vertex cover of size $m-k+1$ and none of size $m-k$. 
Let $t$ be the least integer such that $t\ge m$ and $\bar G$ has an 
induced subgraph $H$ on $t$ vertices with no vertex cover of size $m-k$. 
We claim that $t=m$. 
Indeed, if $t>m$, then the minimality of $t$ implies that $H\backslash \{v\}$ 
has a vertex cover of size at most $m-k$ for every $v\in V(H)$. 
Thus by Lemma \ref{coverlemma}
$t=|V(H)|\le 2(m-k+1)\le m<t$, a contradiction. Thus $t=m$.
By hypothesis, the graph $\bar H$ has a clique $Q$ of size
$k$, but $V(H)-Q$ is a vertex cover of $H$ of size $m-k$, a
contradiction.~\hfill$\Box$

\begin{thm}
\label{dbox}
Let $d\ge 1$ and $m,k\ge 2$ be integers with $k\le m\le 2k-2$, and
let $S$ be a $(k,m)$-agreeable $d$-box society with $n$ voters. Then
the agreement number of $S$ is at least $n-m+k$, and this bound is
best possible.
\end{thm}

\begin{proof}
The agreement graph $G$ of $S$ satisfies the hypothesis of
Theorem~\ref{smallm}, and hence it has a clique of size at least
$n-m+k$ by that theorem.  Since $d$-box societies satisfy the
conclusion of Fact~\ref{fac1}, the first assertion follows.
The bound is best possible, because the graph consisting of a clique
of size $n-m+k$ and $m-k$ isolated vertices is an interval graph.
\end{proof}

%--------------------------------------
\section{Discussion}%
\label{sec:conc}

As we have seen, set intersection theorems can provide
a useful framework to model and understand the relationships
between sets of preferences in voting, and this context leads to new mathematical questions.
We suggest several directions which the reader may wish to explore.

Recent results in discrete geometry have social
interpretations.  The piercing number \cite{katchalski} of approval
sets can be interpreted as the minimum number of platforms that are
necessary such that everyone has some platform of which he or she
approves.  
Set intersection theorems on other spaces (such as trees and cycles)
are derived in \cite{NRSu06} as an extension of both Helly's theorem
and the KKM lemma \cite{kkm};
as an application the authors show that in a super-agreeable society with a circular political spectrum, there must be a platform that has the approval of a strict majority of voters (in contrast with Theorem \ref{cor:super-agreeable}).  Chris Hardin \cite{hardin}
has recently provided a generalization to $(k,m)$-agreeable societies
with a circular political spectrum.

What results can be obtained for other spectra?
The most natural problem seems to be to determine the agreement
proportion for $\R^d$-convex and $d$-box $(k,m)$-agreeable
societies. The smallest case where we do not know the answer is
$d=2$, $k=2$, and $m=3$.  Rajneesh Hegde (private communication)
found an example of a $(2,3)$-agreeable $2$-box society with
agreement proportion $3/8$.
There may very well be a nice formula, because
 for every fixed integer $d$ the agreement number 
of a $d$-box society can be computed in polynomial time.
This is because the clique number of the corresponding agreement graph
(also known as a graph of {\em boxicity} at most $d$) can be determined by an easy 
polynomial-time algorithm.
On the other hand, for every $d\ge2$ it is NP-hard to decide whether 
an input graph has boxicity at most $d$~\cite{krat,Yan}. 
(For $d=1$ this is the same as testing whether a graph is an interval
graph, and that can be done efficiently.)

Passing from results about platforms in societies to results about
a finite set of candidates appears to be tricky in dimensions
greater than 1.  Are there techniques or additional hypotheses that
would give useful results about the existence of candidates who have
strong approval in societies with multi-dimensional spectra? 

We may also question our assumptions.  
While convexity seems to be a rational assumption for
approval sets in the linear case, 
in multiple dimensions additional considerations may become important.
One might also explore the possibility of disconnected approval sets:
what is the agreement proportion of a $(k,m)$-agreeable society in which
every approval set has at most two components?

One might also consider varying levels of agreement.  For instance, in a $d$-box society,
two voters might not agree on every axis, so their approval sets do not intersect, 
but it might be the case that many of the projections of their approval sets do.  In
this case, one may wish to consider an agreement graph with weighted
edges.

Finally, we might wonder about the agreement parameters $k$ and $m$
for various real-world issues.  For instance, a
society considering outlawing murder would probably be much more
agreeable than that same society considering tax reform.  
Currently, we can empirically measure these
parameters only by surveying large numbers of people about their
preferences.  It is interesting to speculate about methods for
estimating suitable $k$ and $m$ from limited data.

This article grew out of the observation that Helly's theorem, a
classical result in convex geometry, has an interesting voting
interpretation.
This led to the development of mathematical questions and theorems whose 
interpretations yield
desirable conclusions in the voting context, e.g., Theorems
\ref{thm:km-agree}, 
\ref{thm:clique}, \ref{dconvex}, and \ref{dbox}. 
It is nice to see that when classical theorems 
have interesting social interpretations, the social context can also
motivate the study of new mathematical questions.

\section{Acknowledgements}
The authors wish to thank a referee for many valuable suggestions,
and to the \textsc{Monthly} Editor for careful reading and detailed
comments that improved the presentation of the paper.
The authors gratefully acknowledge support
 by these NSF Grants:
Berg by DMS-0301129, Norine by DMS-0200595, Su by DMS-0301129 and DMS-0701308,
Thomas by DMS-0200595 and 0354742.

%--------------------------------------

\begin{footnotesize}

\bibliographystyle{plain}

\bigskip
\bigskip

\noindent
\textbf{\uppercase{Deborah E. Berg}} received a B.S. in Mathematics from Harvey Mudd College in 2006    
and an M.S. in Mathematics from the University of Nebraska-Lincoln in 2008.Ê She is currently working towards her Ph.D. in Mathematics Education.Ê When      
not busy studying, doing graph theory, or teaching, she spends her time         
practicing Shotokan karate.\\
\textit{Department of Mathematics, University of Nebraska, Lincoln NE 68588}\\
  {\tt debbie.berg@gmail.com}
  
  \bigskip
  
\noindent
\textbf{\uppercase{Serguei Norine}} received his M.Sc. in mathematics from St.~Petersburg            
State University, Russia in 2001, and Ph.D. from Algorithms,                    
Combinatorics and Optimization program at Georgia Institute of                  
Technology in 2005. After a brief stint in finance following his                
graduation, he returned to academia and is now an instructor at                 
Princeton University. His interests include graph theory and combinatorics.\\
  \textit{Department of Mathematics, Fine Hall, Washington Road, Princeton NJ, 08544}\\
  {\tt snorin@math.princeton.edu}

  \bigskip

\noindent
\textbf{\uppercase{Francis Edward Su}} 
received his Ph.D. at Harvard in 1995, and is now a Professor of Mathematics at Harvey Mudd College.  From the MAA, he received the 2001 Merten M. Hasse Prize for his writing and the 2004 Henry L. Alder Award for his teaching, and he was the 2006 James R.C. Leitzel Lecturer.  
He enjoys making connections between his research in geometric combinatorics and applications to the social sciences.  He also authors the popular
\textit{Math Fun Facts} website.\\
  \textit{Department of Mathematics, Harvey Mudd College, Claremont CA 91711}\\ 
  {\tt su@math.hmc.edu}

  \bigskip

\noindent
\textbf{\uppercase{Robin Thomas}} received his Ph.D. from Charles University in Prague,                
formerly Czechoslovakia, now the Czech Republic. He has worked                  
at the Georgia Institute of Technology since 1989. Currently he                 
is Professor of Mathematics and Director of the multidisciplinary               
Ph.D. program in Algorithms, Combinatorics and Optimization. In                   
1994 he won, jointly with Neil Robertson and Paul Seymour, the
 D. Ray Fulkerson prize in Discrete Mathematics.\\
  \textit{School of Mathematics, Georgia Institute of Technology, Atlanta, Georgia 30332}\\
  {\tt thomas@math.gatech.edu}

  \bigskip

\noindent
\textbf{\uppercase{Paul~Wollan}} received his Ph.D. in Georgia Tech's interdisciplinary              
Algorithms, Combinatorics, and Optimization program in 2005.  He spent a        
year as a post doc at the University of Waterloo, and he is currently a         
Humboldt Research Fellow at the University of Hamburg.\\
  \textit{Mathematisches Seminar der Universit\"at Hamburg, Bundesstr. 55, D-20146 Hamburg
Germany}\\
  {\tt wollan@math.uni-hamburg.de}

\end{footnotesize}

\end{document}